\documentclass{amsart}

\usepackage{amsmath, amsthm, amssymb, amsfonts, amscd, graphicx, endnotes, tikz, color, hyperref,fancyvrb,longtable,array,microtype}
\usepackage[mathscr]{eucal}
\usepackage{mathtools}
\usepackage{tikz-cd}

\usetikzlibrary{arrows}

\def\C{{\mathbb C}}

\def\N{{\mathbb N}}

\def\Q{{\mathbb Q}}

\def\Q{{\mathbb Q}}
\def\R{{\mathbb R}}

\def\Z{{\mathbb Z}}

\def\0{{\mathbf 0}}
\def\1{{\mathbf 1}}

\def\Ocal{{\mathcal O}}

\def\fp{{\mathfrak p}}

\def\ab{\mathrm{ab}}
\def\Gal{\mathrm{Gal}}

\def\max{\mathrm{max}}

\def\ab{\mathrm{ab}}

\theoremstyle{plain}

\newtheorem{thm}{Theorem}

\newtheorem{lem}[thm]{Lemma}
\newtheorem{mylem}[thm]{Lemma}

\theoremstyle{definition}

\newcommand{\Bomb}{\cite{bombierigubler}}

\title[Totally p-adic Numbers of Degree $3$]{Totally p-adic Numbers of Degree $3$}

\author{Emerald Stacy}

\address{Emerald Stacy; Department of Mathematics and Computer Science; Washington College; 300 Washington Avenue; Chestertown MD 21620 U.S.A.}

\email{estacy2@washcoll.edu}

\begin{document}
\maketitle
\section{Introduction}

Recall that an algebraic number $\alpha$ is \textbf{totally $p$-adic} (respectively totally real) if the minimal polynomial of $\alpha$, $f_\alpha\in \Q[x]$, splits completely over $\Q_p$ (respectively $\R$). We will denote by $h(\alpha)$ the logarithmic Weil height of $\alpha$ \Bomb.

In 1975, Schinzel used the arithmetic-geometric mean inequality to prove that if $\alpha$ is a totally real algebraic integer, with $\alpha \neq 0, \pm 1$, then 
\[h(\alpha) \geq \tfrac{1}{2}\log \left(\tfrac{1+\sqrt{5}}{2}\right)\]
with equality if $\alpha =\frac{1+\sqrt{5}}{2}$ \cite{schinzel1975addendum}.  In 1993, H{\"o}hn \& Skoruppa used an auxiliary function to provide an alternate proof of Schinzel's bound \cite{MR1251237}.
Bombieri \& Zannier \cite{bombieri2001note} proved that an analogue to Schinzel's Theorem holds in $\Q_p$ for each prime $p$, although the analogous best possible lower bound is unknown. 

Additionally, there have been some results constructing totally $p$-adic (or totally real) algebraic numbers of small height.  In particular, these results provide an upper bound on the smallest height attained by $\alpha$ under certain splitting conditions. The degree of a totally $p$-adic number is the degree of its minimal polynomial with coefficients in $\Z$.
Petsche \cite{petschepreprint} proved that for odd primes $p$, there exists some totally $p$-adic $\alpha \in \overline{\Q}$ of degree $d \leq p-1$, and 
\[ 0 < h(\alpha)\leq \tfrac{1}{p-1}\log\left(\tfrac{p+\sqrt{p^2+4}}{2}\right).\]
Recently, Pottmeyer \cite{pottmeyer2018small} has improved upon Petsche's upper bound, and obtained the existence of totally $p$-adic $\alpha$ such that
\[ 0 < h(\alpha) \leq \tfrac{\log p }{p} .\]
 
In 1980, Smyth created a set of totally real numbers of small height by taking all preimages of $1$ under the map $\phi(x) = x - \tfrac{1}{x}$.  The heights of the points in this set have a limit point $\ell \approx 0.27328$ \cite{smyth1980measure}.  In \cite{petsche2019dynamical}, Petsche \& Stacy use an argument inspired by this result of Smyth to provide an upper bound on the smallest limit point of heights of totally $p$-adic numbers of degree $d$.

In this paper, we fix the degree $d$ to be $3$ and let the prime $p$ vary.  In particular, we define $\tau_{d,p}$ to be the smallest  height attained by a totally $p$-adic, nonzero, non-root of unity, algebraic number of degree $d$.  For any pair $d$ and $p$, we know $\tau_{d,p} < \infty$ since we can construct a Newton Polygon for an irreducible polynomial of degree $d$ that splits completely over $\Q_p$ \cite{cassels1986local}.

In this paper, we develop tools to determine $\tau_{3,p}$ for all $p \geq 5$.    
In Section \ref{section1}, we develop and prove an algorithm to determine $\tau_{3,p}$ for a given prime $p$, which we implement in Section \ref{section2}. All code was written for SageMath, Version $8.2$ \cite{sagemath}, and is included within Section \ref{section2}. A table of results can be found in Section \ref{results}, and Section \ref{conclusion} describes future areas of interest.

\section{The Algorithm}\label{section1}
In Section \ref{abel}, we prove that $\tau_{3,p} \leq 0.70376$ for all $p \geq 5$. To do so, we establish that for every prime $p$, there is a cubic polynomial with an abelian Galois group that splits completely over $\Q_p$. 
By the height-length bound \cite[Proposition $1.6.7$]{bombierigubler}, a list of all cubic polynomials with length less than $68$ will contain all irreducible, non-cyclotomic, cubic polynomials with roots of height less than $0.70376$.
By the Northcott property there are only finitely many such polynomials, and thus we have a finite list to check for $\tau_{3,p}$ and our algorithm will terminate. 

In Section \ref{cardano}, we use the method of Cardano to determine the roots of a cubic polynomial. In Sections \ref{1mod3} and \ref{2mod3}, we establish criteria to determine if those roots are in $\Q_p$. The criteria are different depending if $p \equiv 1 \pmod 3$ or $p \equiv 2 \pmod 3$, since $\Q_p$ contains a primitive cube root of unity if and only if $p \equiv 1 \pmod 3$. In Section \ref{section2}, we implement the algorithm, the results of which can be found in Section \ref{results}.

\subsection{Establishing Termination}\label{abel}
To establish that our algorithm will terminate, we create a finite list of polynomials, and verify that for each prime, there must be a polynomial in our list that will split completely over $\Q_p$.

Let $f_\alpha$ denote the minimal polynomial of $\alpha$.
Then $h(\alpha) = \frac{1}{3} \log  M(f_\alpha)$, where $M(f_\alpha)$ is the Mahler measure of $f_\alpha$. Thus, if $M(f_\alpha) \leq 8.5$, then $h(\alpha) \leq 0.71335$.
The function \textbf{mahler\_measure\_cubic} calculates the Mahler measure of the cubic polynomial $f(x) = ax^3+bx^2+cx+d$.  
\begin{small}
\begin{Verbatim}[frame=single]
def mahler_measure_cubic(a,b,c,d):
    M = a
    Poly = a*x^3 + b*x^2 + c*x + d
    Roots = Poly.roots(CC)
    for i in [0..len(Roots)-1]:
        M = M * max(1,abs(Roots[i][0]))
    return M.n(digits=10)
\end{Verbatim}
\end{small}

For $f(x) = \sum_{i=0}^d a_i x^i,$  the \textbf{length} of $f$ is $L(f) = \sum_{i=0}^d |a_i|$. The length will be useful to us since for any polynomial $f$, 
\[L(f) \leq 2^d M(f),\]  
where $d = \deg f$ \cite[Proposition $1.6.7$]{bombierigubler}.
Thus, the following program generates a list of all cubic polynomials with 
\[L(f) \leq 2^3 (8.5) = 68\]
and remove any polynomial that is either reducible or has Mahler measure greater than $8.5$.
We use the built-in Sage function \textbf{is\_irreducible()} to determine if a polynomial is irreducible over $\Q$.

In addition to the polynomial and Mahler measure, the list also stores the coefficients of the cubic in its so-called depressed form ($x^3+Ax+B$), the discriminant of the polynomial, and the height of the roots.  For more information on depressing a cubic, please see Section 2.2.

The command \textbf{sorted()} will reorganize the array in ascending order of the first value - in this case it will sort by Mahler measure, which is equivalent to sorting by height.
The output of this program is $26796$ polynomials that is saved as the file \textbf{irred\_polynomials\_L68}. Runtime was $124$ minutes.

\begin{small}
\begin{Verbatim}[frame=single]
R.<x> = QQ[]
Polynomials=[]
L=68
for a in [1..L]:
    for b in [-L+abs(a)..L-abs(a)]:
        for c in [-L+abs(a)+abs(b)..L-abs(a)-abs(b)]:
            for d in [-L+abs(a)+abs(b)+abs(c)..L-abs(a)-abs(b)-abs(c)]:
                Poly = a*x^3 + b*x^2 + c*x + d
                if Poly.is_irreducible()==True:
                    MM = mahler_measure_cubic(a,b,c,d)
                    A = (3*a*c - b^2 ) / (3*a^2 )
                    B = (27*a^2*d - 9*a*b*c + 2*b^3 ) / (27*a^3 )
                    Delta = B^2 + 4 * A^3 / 27
                    h = 1/3 * log(MM);
                    if MM <= L/8:
                        Polynomials.append([MM,a,b,c,d,A,B,Delta,h])
Polynomials=sorted(Polynomials)
\end{Verbatim}
\end{small}

Next, we remove from this list all polynomials with non-abelian Galois group.
In general, the Galois group of a polynomial $f(x) \in \Z[x]$ of degree $d$ is isomorphic to a subgroup of $A_d$ if and only if the discriminant of $f$ is a square in $\Q$ \cite[Theorem $1.3$]{ConradGGCQ}.  In the case of $f$ cubic, the Galois group of $f$ is $A_3$, and thus abelian, if and only if the discriminant of $f$ is a square in $\Q$.

Let $K$ be the number field created by adjoining the roots of $f$ to $\Q$ and let $\Delta$ be the discriminant of $K$. By the Kronecker-Weber Theorem, $K$ must be contained within a cyclotomic extension of $\Q$.   Let $m$ be the conductor of $K$, meaning the smallest $m$ such that $K$ is a subfield of $\Q(\zeta_m)$, where $\zeta_m$ is a primitive $m^\text{th}$ root of unity. 
To calculate the conductor, we turn to a special case of the Hasse Conductor-Discriminant formula, as follows.

\begin{thm}\cite[Theorem 6]{MR1545136}\label{Hasse}
Let $K$ be an abelian extension of $\Q$, with $[K:\Q]=3$ and discriminant $\Delta$.  Let $p_1, p_2, \dots, p_n$ be all the primes (aside from $3$) that divide $\Delta$.  If $3$ divides $\Delta$, then the conductor of $K$ is  $9 p_1 p_2\dots p_n$.  If $3$ not does divide $\Delta$, then the conductor of $K$ is $p_1 p_2 \dots p_n$.
\end{thm}

The following program begins by identifying if each cubic polynomial has an abelian Galois group. If so, then the program calculates the discriminant of $K$ (the number field obtained by adjoining the roots of $f$ to $\Q$)
by applying the built-in function \textbf{absolute\_discriminant()}.
It then applies Theorem \ref{Hasse} and uses the built in Sage command \textbf{factor()} to determine the conductor of $K$.  All of this output is stored in the array \textbf{AbelianCubics}, which contains the information for $156$ polynomials.
\begin{small}
\begin{Verbatim}[frame=single]
Polynomials=load('irred_polynomials_L68')
L=len(Polynomials)
AbelianCubics=[]

for i in [0..L-1]:
    Poly = Polynomials[i];
    a = Poly[1];
    b = Poly[2];
    c = Poly[3];
    d = Poly[4];
    D = b^2*c^2 - 4*a*c^3 - 4*b^3*d - 27*a^2*d^2 + 18*a*b*c*d;
    
    if D.is_square()==True:
        K.<j> = NumberField(a*x^3 + b*x^2 + c*x + d)
        DD = K.absolute_discriminant()
        MM = Poly[0];
        h = Poly[8];
        Factors = DD.factor()
        list_of_factors = list(Factors)
        L = len(list_of_factors)
        Cond = 1
        
        for i in [0..L-1]:
            Cond = Cond*list_of_factors[i][0]
            if list_of_factors[i][0]==3:
                Cond = Cond*3
        C = Cond
        AbelianCubics.append([h, a*x^3 + b*x^2 + c*x + d ,DD,C]);
\end{Verbatim}
\end{small}

The following lemma is well known, but for lack of a convenient reference, we provide a proof.

\begin{lem}\label{frogger}
Let $\alpha \in \Q(\zeta_n)$ have minimal polynomial $f_\alpha \in \Z[x]$, and let
\[G_\alpha = \{ [i] \in (\Z/n\Z)^\times \mid \sigma_i(\alpha)=\alpha\},\]
where $\sigma_i(\zeta_n) = \zeta_n^i$.
Thus $G_\alpha$ is the subgroup of $(\Z/n\Z)^\times$ corresponding to \\
$\Gal(\Q(\zeta_n)/\Q(\alpha))$ via the isomorphism $(\Z/n\Z)^\times \cong \Gal(\Q(\zeta_{n}) / \Q)   $.
 Let $p \nmid n$ be a prime.  Then $f_\alpha$ splits completely in $\Q_p$ if and only if $[p] \in G_\alpha$.
\end{lem}

\begin{proof}
The automorphism $\sigma_p\in \Gal(\Q(\zeta_n)/\Q)$ satisfies $\sigma_p(x)\equiv x\pmod{p}$ for all $x\in\Z[\zeta_n]$ \cite[Lemma 4.51]{BakerANT}.  Since $\Q(\zeta_n)/\Q$ is an abelian extension, $\Q(\alpha)/\Q$ is a Galois extension and therefore $\sigma_p$ restricts to an automorphism $\sigma_p\in\Gal(\Q(\alpha)/\Q)$; the above congruence implies that $\sigma_p$ is the Frobenius element of $\Gal(\Q(\alpha)/\Q)$ associated to the prime $p$.

If $[p]\in G_\alpha$, then $\sigma_p$ is the identity element of $\Gal(\Q(\alpha)/\Q)$, which implies that $p$ splits completely in 
$\Q(\alpha)$ \cite[Proposition 4.36]{BakerANT}; 
that is $p\Ocal_{\Q(\alpha)}=\fp_1\dots\fp_d$, 
where $d=[\Q(\alpha):\Q]$.  It follows that each local degree $e(\fp_i/p)f(\fp_i/p)=[\Q(\alpha)_{\fp_i}:\Q_p]$ is equal to $1$ \cite[Theorem 5.25]{BakerANT},
 which means that $\Q(\alpha)_{\fp_i}=\Q_p$ for $i=1,2,\dots,d$.  In particular, $\Q(\alpha)\subseteq\Q_p$, and therefore as $\Q(\alpha)/\Q$ is Galois, all $d$ of the Galois conjugates of $\alpha$ are in $\Q_p$ as well.  
 Hence $f_\alpha(x)$ splits completely in $\Q_p$.  The converse follows from a straightforward reversal of this argument.
\end{proof}

For each polynomial $f_{\alpha}$ in \textbf{AbelianCubics}, we want to determine the congruence classes modulo $m$ of a prime $p$ for $f_{\alpha}$ to split completely in $\Q_p$, where $m$ is the conductor of the splitting field of $f_\alpha$.  
The following code goes through each line in the array \textbf{AbelianCubics}, and for each polynomial $f_\alpha$ in
the list, computes the set $B_\alpha \subseteq (Z/mZ)^\times$ so
that $f_\alpha$ splits completely in $\Q_p$ if and only if $[p] \in B_\alpha$, where $[p]$
denotes the residue of $p \pmod m$. 

Note that if $(\Z/m\Z)^\times$ has a unique index $3$ subgroup, then this group must be $G_\alpha$.  
In the case that $(\Z/m\Z)^\times$ does not have a unique index $3$ subgroup, we check the first $50$ primes to determine if there is a root in $\Q_p$ via Hensel's Lemma.  
When a root of $f_\alpha$ is determined to be in $\Q_p$, we know that for all primes $q$ with $q \equiv p \pmod m$, $f_\alpha$ must split completely in $\Q_p$, by Lemma \ref{frogger}.  Further, we know there are $|(\Z/m\Z)^\times|/3$  congruence classes for which $f_\alpha$ splits completely in $\Q_p$.  Thus, after testing the first $50$ primes, the code checks the cardinality of the set of congruences to ensure all were found. For this particular list of polynomials, $50$ is sufficient to identify the index $3$ subgroup.

\begin{small}
\begin{Verbatim}[frame=single]
AbelianCubics=load('AbelianCubics')
L=len(AbelianCubics);
P = Primes();

for i in [0..L-1]:
    Poly = AbelianCubics[i][1]
    PolyList = Poly.list()
    a = PolyList[3]
    b = PolyList[2]
    c = PolyList[1]
    d = PolyList[0]
    Cond = AbelianCubics[i][3]
    v = [1];    
    for j in [0..50]:
        for k in [1..P[j]-1]:
            M = Integer( a*k^3 + b*k^2 + c*k + d )
            M = M%P[j]
            N = Integer( 3*a*k^2 + 2*b*k + c )
            N = N%P[j]
            if M==0 and N>0:
                v.append(P[j]%Cond)               
    V = sorted(v)
    V = set(V)
\end{Verbatim}
\end{small}

The results of this code are included as a supplement to this paper in a file called  \textbf{Degree3Table.pdf}. A sampling of the data is included here for reference.
\renewcommand*{\arraystretch}{1.2}
\begin{center}
\begin{table}[h]
\begin{tabular}{|c|c|l|}
\hline
$f_\alpha$ & $h(\alpha) $ & $\alpha$ is totally $p$-adic if and only if\\
\hline
$x^3-x^2-2x+1$ & $0.26986 $ & $p \equiv 1, 6 \pmod 7$\\
\hline
$x^3-3x^2+1$ & $ 0.35252 $ & $ p \equiv 1,8 \pmod 9$\\
\hline
$ 3x^3 -4  x^2-5  x +3   $ & $ 0.60981 $ & $p \equiv  1,3,8,9, 11,  20,23, 24,   27, 28, 33, 34, 37, 38,$\\
& & $  41, 50, 52, 53, 58 ,  60   \pmod {61 }$\\
\hline
$ 3x^3 -  x^2 -8 x+3    $ & $ 0.69106 $ & $p \equiv   1, 3, 7, 8, 9, 10, 17, 21, 22, 24, 27, 30, 43,$\\
& & $ 46, 49,  51, 52, 56, 63, 64, 65, 66, 70, 72   \pmod {73 }$\\
\hline
$ 2x^3 -9  x^2+3  x+2    $ & $ 0.69903 $ & $p \equiv    1, 2, 4, 8,  16, 31, 32,47,55, 59, 61, 62 \pmod { 63}$\\
\hline
$ x^3   -9x^2 +6 x+1    $ & $ 0.70376 $ & $p \equiv  1, 5,  8, 11,23, 25,38, 40, 52, 55,  58 ,62  \pmod { 63}$\\
\hline
\end{tabular}
\caption{\label{mytable}}
\end{table}
\end{center}
\newpage
\begin{thm}\label{N3}
Let $p$ be a prime.  Then $\tau_{3,p} \leq 0.70376$.
\end{thm}

\begin{proof}  For a prime $p$, denote $\tau_{3,p}^\ab$ be the smallest nontrivial height of an abelian, cubic, totally $p$-adic number. Note that $\tau_{3,p} \leq \tau_{3,p}^\ab$. Thus, if we show that $\tau_{3,p}^\ab \leq 0.70376$, we have proven the theorem.

Based on the results 
from Table \ref{mytable}, we know
\[\tau_{3,3}^\ab \leq 0.609817669, \text{ and }\]
\[\tau_{3,7}^\ab \leq 0.501878627.\]
All primes $p \neq 3, 7$, when reduced modulo $63$, are contained in $(\Z/63\Z)^\times$.  Observe that 
\[\left(\Z/63\Z\right)^\times = \{1,2,4,5,8,10,11,13,16,17,19,20,22, 23,25, 26, 29, 31, 32, 34, 37,\]
\[  38, 40, 41, 43, 44, 46, 47,50, 52, 53, 55, 58, 59, 61, 62\}.  \]
Further, we observe that 
\[ \tau_{3,p}^\ab \leq 
\begin{cases}
0.269862305 & \text{ if } p \equiv 1, 6 \pmod 7, \\
 0.352525605 & \text{ if } p \equiv 1, 8 \pmod 9 \\
\end{cases}
  \]
  Thus
  \[\tau_{3,p}^\ab \leq 0.269862305 \text { for } p \equiv 1, 8, 13, 20, 22, 29, 34, 41, 43, 50 , 55, 62 \pmod {63}, \text{ and } \]
    \[\tau_{3,p}^\ab \leq 0.352525605 \text { for } p \equiv 10, 17, 19, 26, 37, 44, 46, 53 \pmod {63}. \]
It remains to determine an upper bound on $\tau_{3,p}^\ab$ for 
\[ p \equiv 2,4,5,11,16, 23,25,  31, 32,  38, 40,  47,  52, 58, 59, 61 \pmod {63}.\]
Note that each of the above numbers falls into one of the following two sets:
\[p \equiv   1, 2, 4, 8, 16,  31,32, 47, 55, 59, 61, 62   \pmod { 63}\]
\[p \equiv  1, 5,  8, 11,23, 25,38, 40, 52,   55 ,58, 62   \pmod { 63}.\]
Further, we observe that by the last two lines of Table \ref{mytable}, given any prime $p$, one of the polynomials in the table must split completely over $\Q_p$.  
\end{proof}

\subsection{Determining Roots of Cubic Polynomials}\label{cardano}
In \emph{Ars Magna}, Cardano describes a method to find the roots of a cubic polynomial $f$ as elements of $\C$ \cite{cardano1968ars}.  This method is analogous to completing the square for a quadratic polynomial.    
We use Cardano's method to determine if a cubic polynomial in $K[y]$ splits completely over $K$, where $K$ is an arbitrary field of characteristic not equal to $2$ or $3$.
Beginning with an arbitrary cubic polynomial in $K[y]$,
\begin{equation*}
g(y) = ay^3+by^2+cy+d
\end{equation*}
we divide through by the leading coefficient
and perform a change of variables $y=x-\tfrac{b}{3}$ to eliminate the quadratic term, yielding a
monic depressed cubic polynomial with coefficients in $K$,
\[f(x) = x^3+Ax+B.\]
Note that since the transformations to depress the cubic simply shift the roots by $\tfrac{b}{3a}$, so $g$ splits over $K$ if and only if $f$ splits over $K$.

\begin{mylem}[Cardano]\cite{cardano1968ars}\label{roots}
Let $L$ be an algebraically closed field of characteristic not equal to $2$ or $3$, and let $\zeta$ be a primitive cube root of unity in $L$.  Let $f(x) = x^3+Ax+B \in L[x]$, and let $\Delta = B^2+4A^3/27$. 
If $A=0$, let $C=-B$, and if $A \neq 0$, let $C$ be either square root of $\Delta$ in $L$.
Let $u$ be a cube root of $\frac{-B+C}{2}$ and let $v = -\frac{A}{3u}$.  Then the roots of $f$ are $u+v, \zeta u + \zeta^2 v$, and $ \zeta^2 u + \zeta v.$
\end{mylem}
  
To determine when a cubic polynomial $f(x) \in \Q_p[x]$ splits completely over $\Q_p$, the method will depend on whether $\Q_p$ contains a primitive cube root of unity, which happens exactly when $p \equiv 1 \pmod 3$.  Thus, we consider two cases: $p \equiv 1 \pmod 3$ and $p \equiv 2 \pmod 3$.  

\subsection{Case 1}\label{1mod3}
Suppose $p \equiv 1 \pmod 3$.

\begin{thm}\label{onemodthreethm} 
Let $K$ be a field of characteristic not equal to $2$ or $3$, let $L$ be an algebraic closure of $K$, and assume that $K$ contains a primitive cube root of unity, $\zeta$.  Let $f(x) = x^3+Ax+B \in K[x]$, and $\Delta = B^2+4A^3/27$.  If $A=0$, let $C=-B$, and if $A\neq 0 $, let $C$ be either square root of $\Delta$ in $L$.
  Then $f$ splits completely over $K$ if and only if
\begin{itemize}
\item[\textbf{(a)}] $\Delta$ is a square in $K$, and 
\item[\textbf{(b)}] $\frac{-B+C}{2}$ is a cube in $K$.
\end{itemize}
\end{thm}

\begin{proof}
Suppose $A=0$.  Then $\Delta = B^2$ is a square in $K$, so \textbf{(a)} is true.
Additionally, $C=-B$ and $f(x) =x^3+B$, which splits completely over $K$ if and only if $-B$ is a cube in $K$, which happens exactly when \textbf{(b)} holds.

Now suppose $A \neq 0$.  Let $u$ be a cube root of $\frac{-B+C}{2}$ and let $v = -\frac{A}{3u}$.
Let $F$ be a Galois extension of $K$ containing $C$ and $u$.

Suppose the conditions \textbf{(a)} and \textbf{(b)} are met.
By Lemma \ref{roots}, the roots of $f$ are $u+v,  \zeta u + \zeta^2 v, $ and $\zeta^2 u + \zeta v$
and thus $f$ splits completely over $K$.

Conversely, suppose that $f$ splits completely over $K$.  Let $\sigma \in \Gal(L/K)$.
Since $\sigma$ fixes $u+v$ and $\zeta u+\zeta^2 v$,
\begin{equation}\label{lin}
u+v = \sigma(u)+\sigma(v), \text{ and } %\hspace{2cm}
\zeta u + \zeta^2 v = \zeta \sigma(u) + \zeta^2 \sigma(v).
\end{equation}
Note that 
$ \left( \begin{smallmatrix}1 & 1 \\ \zeta & \zeta^2\end{smallmatrix} \right)$ has a non-zero determinant and thus 
\begin{equation}\label{blah}
 \begin{pmatrix}1&1 \\ \zeta & \zeta^2\end{pmatrix} \begin{pmatrix} x \\ y \end{pmatrix} =  \begin{pmatrix}\sigma(u)+\sigma(v) \\ \zeta \sigma(u) + \zeta^2 \sigma(v)\end{pmatrix} 
 \end{equation}
 has a unique solution.  By (\ref{lin}), $x=u, y=v$ is a solution to (\ref{blah}) and $x=\sigma(u)$, $y = \sigma(v)$ is a solution to (\ref{blah}) as well.  Therefore $u = \sigma(u)$.  
 By the Galois correspondence, $u \in K$, and thus \textbf{(b)} holds.  Thus $u^3 = \frac{-B+C}{2} \in K$.  Since $C=2u^3+B$, $C \in K$ and therefore $\Delta = B^2+4A^3/27 = C^2$ is a square in $K$, and \textbf{(a)} is true.
\end{proof}

\begin{lem}\label{lem2}
Let $p$ be a prime, $p \neq 3$, and 
let $a \in \Z_p$ with $|a|_p = 1$.  
Then $a$ is a cube in $\Q_p$ if and only if $a \pmod p$ is a cube in $\Z_p/p\Z_p$.
\end{lem}

\begin{proof}
Suppose that $a$ is a cube in $\Z_p$.  Then $a$ is a cube in $\Z_p / p \Z_p$ by the nature of quotient rings.  

Conversely, suppose $a_0$ is a cube in $\Z/p\Z$ where $a_0 \equiv a \pmod p$, and let $b_0 \in \Z/p\Z$ satisfy $b_0^3 \equiv a_0 \pmod p$.
Let $f(x) = x^3-a$.  Note that $p \nmid 3, b_0$. 
By the strong triangle inequality, 
\begin{align*}
\left|f(b_0)\right|_p &= \left|b_0^3 - a \right|_p \\
%&= \left|b_0^3 - a_0 + a_0- a \right|_p \\ 
&\leq  \max \left\{\left|b_0^3 - a_0\right|_p, \left| a_0- a \right|_p\right\} \\ 
&\leq \tfrac{1}{p}.
\end{align*}
Further,
\[
\left|f'(b_0)\right|_p = \left|3b_0^2 \right|_p = 1.
\]
By Hensel's Lemma, $a$ is a cube in $\Q_p$.
\end{proof}

\begin{thm}\label{Thm1mod3} Let $p$ be a prime, with $p \equiv 1 \pmod 3$.  Then the following algorithm yields $\tau_{3,p}$. 
\begin{itemize}
\item[\textbf{(1)}] Create a list, in ascending order of Mahler measure, of all irreducible, non-cyclotomic cubic polynomials in $\Z[x]$ with Mahler measure bounded above by $8.5$.   Let $f(x)$ be the first polynomial on the list. 
\item[\textbf{(2)}] Convert $f(x)$ into depressed form $g(x) = x^3+ Ax+B$ and let $\Delta =B^2+4A^3/27 $.  
\item[\textbf{(3)}] If $\Delta$ is not a square in $\Q_p$, return to step {\bf (2)} with the next polynomial on the list.
\item[\textbf{(4)}] If $A=0$, let $C=-B$, and otherwise let $C$ be a square root of $\Delta$ in $\Q_p$.  If $\frac{-B+C}{2}$ is not a cube in $\Q_p$, return to step {\bf (2)} with the next polynomial on the list.  Otherwise, terminate, $\tau_{3,p} = \tfrac{1}{3} \log M(f)$.  
\end{itemize}
\end{thm}

\begin{proof}
Since $\tau_{3,p} \leq \tau_{3,p}^\ab$, by Theorem \ref{N3} we know that $\tau_{3,p} \leq 0.70376$.  By \cite[Proposition $1.6.7$]{bombierigubler}, a list of all polynomials with length less than $68$ will contain all irreducible, non-cyclotomic, cubic polynomials with Mahler measure bounded above by $8.5$.
Any degree $3$ algebraic number of height less than or equal to $0.70376$ will be a root of a polynomial in the list.  Thus, this algorithm will always terminate successfully.

Let $f$ be the polynomial being considered.
By Theorem \ref{onemodthreethm}, steps \textbf{(3)} and \textbf{(4)} will detect exactly when $f$ splits completely over $\Q_p$.
\end{proof}
\subsection{Case 2}\label{2mod3}
Suppose $p \equiv 2 \pmod 3$.
\begin{thm}\label{twomodthreethm} 
Let $K$ be a field of characteristic not equal to $2$ or $3$, $K'$ be an algebraic closure of $K$, $\zeta$ be a primitive cube root of unity in $K'$, and assume that $\zeta \notin K$.  Let $f(x) = x^3+Ax+B \in K[x]$ with $B \neq 0$ and let $\Delta = B^2+4A^3/27$.  If $A=0$, let $C=-B$, and if $A \neq 0$, let $C$ be either square root of $\Delta$ in $K'$.
Then $f$ splits completely over $K$ if and only if
\begin{itemize}
\item[\textbf{(a)}] $\Delta$ is a square in $K(\zeta)$ and not a square in $K$, and 
\item[\textbf{(b)}] $\frac{-B+C}{2}$ is a cube in $K(\zeta)$ and not a cube in $K$.
\end{itemize}
\end{thm}

\begin{proof} 
Let $u$ be a cube root of $\tfrac{-B+C}{2}$ and let $v=\tfrac{-A}{3u}$.
By Lemma \ref{roots}, the roots of $f$ are $u+v,  \zeta u + \zeta^2 v, $ and $ \zeta^2 u + \zeta v.$

We first suppose $f$ splits completely in $K$.  Let $L$ be a Galois extension of $K$ that contains $u$ and $\zeta$.
Let $\sigma \in \Gal(L/K(\zeta))$.    We want to show that $\sigma$ must fix $u$.  Since we are assuming that $f$ splits completely over $K$, $\sigma$ must fix $u+v,  \zeta u + \zeta^2 v, $ and $ \zeta^2 u + \zeta v,$
\begin{equation}\label{eqn1}
u+v = \sigma(u)+\sigma(v), 
\end{equation}
\begin{equation}\label{eqn3}
\zeta^2 u + \zeta v = \zeta^2 \sigma(u) + \zeta \sigma(v)  .
\end{equation}
By multiplying $(\ref{eqn1})$ by $\zeta$ and subtracting $(\ref{eqn3})$, we obtain
\begin{equation}\label{eqn5}
(\zeta-\zeta^2) u = (\zeta-\zeta^2) \sigma(u),
\end{equation}
so $\sigma(u) = u$ because $\zeta \neq \zeta^2$.
Thus, since all elements in $\Gal(L/K(\zeta))$ fix $u$, $u$ must be in $ K(\zeta)$. 

It remains show $u \notin K$.  Let $\tau \in \Gal(L/K)$ be such that $\tau$ interchanges $\zeta$ and $\zeta^2$.  We now show that $\tau$ does not fix $u$.
Since the roots of $f$ must all be fixed by $\tau$,
\begin{equation}\label{eqn7}
\zeta u + \zeta^2 v = \zeta^2 \tau(u) + \zeta \tau(v) ,
\end{equation}
\begin{equation}\label{eqn8}
\zeta^2 u + \zeta v = \zeta \tau(u) + \zeta^2 \tau(v).
\end{equation}
By multiplying $(\ref{eqn8})$ by $\zeta$, and subtracting $(\ref{eqn7})$, we obtain
\begin{equation}\label{eqn9}
(1- \zeta) u = (1-\zeta)\tau(v)
\end{equation}
and note that $\tau(v)=u$, so $\tau$ does not fix $u$.  Thus $u \notin K$ and \textbf{(b)} holds.  

Further, $u \in K(\zeta)$, so $u^3 = \tfrac{-B+C}{2} \in K(\zeta)$, and thus $\Delta$ is a square in $K(\zeta)$ since $C \in K(\zeta)$.
Since $K(u)$ is contained within $K(\zeta)$, a quadratic extension of $K$, and $u \notin K$,  it follows that $[K(u):K] = 2$.  For sake of contradiction, suppose $\Delta$ is a square in $K$.  Then $u^3 \in K$, so $[K(u):K]=3$ which is not true.  Thus $\Delta$ is not a square in $K$, and \textbf{(a)} holds.

Conversely, suppose that \textbf{(a)} and \textbf{(b)} are true.  
Note that if $A=0$, then $\Delta$ is a square in $K$, contradicting \textbf{(a)}.
Thus, $A \neq 0$.
Let $\sigma$ denote the non trivial element of Gal$(K(\zeta)/K)$.  Since $\zeta$ and $\zeta^2$ share a degree $2$ minimal polynomial, $\sigma$ must permute $\zeta$ and $\zeta^2$.

By \textbf{(a)} and \textbf{(b)}, $u, u^3 \notin K$ and $u, u^3 \in K(\zeta)$.  Since $u^3$ and $ v^3$ are the roots of $r(z) = z^2+Bz-\frac{A^3}{27}$, we have
$\sigma(u)^3 = \sigma(u^3) = v^3$.  Therefore, either $\sigma(u)=v$, $\sigma(u) = \zeta v$,  or $\sigma(u)= \zeta^2 v$.  

We will now show that $\sigma(u)=v$ by eliminating the other two options by way of contradiction.  We rely on the fact that elements of the Galois group send roots of $f$ to roots of $f$, and that $\sigma^2(u) = u$.
If $ \sigma(u) = \zeta v$, then $u = \zeta^2 \sigma(v)$, and $\sigma(u+v) = \sigma(u) + \sigma(v) = \zeta v +\zeta u$.
Since $\zeta v + \zeta u$ is not a root of $f$, $ \sigma(u) \neq \zeta v$.
If $ \sigma(u) = \zeta^2 v$, then $u = \zeta \sigma(v)$, and $\sigma(u+v) = \zeta^2 u + \zeta^2 v$.  Since $\zeta^2 u + \zeta^2 v$ is not a root of $f$, $ \sigma(u) \neq \zeta^2 v$.

Therefore, $\sigma(u) = v$ and $\sigma(v)=u$.  Thus
\begin{align*}
 \sigma(u+v) &= \sigma(u)+\sigma(v)= v + u ,\\
\sigma(\zeta u + \zeta^2 v)&= \sigma(\zeta u) + \sigma(\zeta^2 v) = \zeta^2 v + \zeta u ,\\
 \sigma(\zeta^2 u + \zeta v) &= \sigma( \zeta^2 u) + \sigma( \zeta v) = \zeta v + \zeta^2 v.
\end{align*}
Since $\sigma$ fixes the roots of $f$, $f$ splits completely in $K$. 
\end{proof}
Let $p \equiv 2 \pmod 3$. 
The third cyclotomic polynomial, $\Phi_3(x) = x^2+x+1$, has discriminant $-3$ and is the minimal polynomial for $\zeta$. 
Since $-3$ is not a square in $\Q_p$, $\Phi_3(x)$
 is irreducible over $\Q_p$, and thus $\Q_p$ does not contain a primitive cube root of unity.  
There are exactly three quadratic extensions of $\Q_p$: $\Q_p(\sqrt{p}), \Q_p(\sqrt{-3}), $ and $\Q_p(\sqrt{-3p})$. 
 Let $K = \Q_p(\sqrt{-3}) = \Q_p(\zeta)$, the unique unramified quadratic extension of $\Q_p$.
The $p$-adic absolute value on $\Q_p$ extends uniquely to $\Q_p(\sqrt{-3})$ by 
\[|a+b\sqrt{-3}|_p =\left| N_{K/\Q_p}(a+b\sqrt{-3})\right|_p^{1/2} = \left| a^2+3b^2\right|_p^{1/2}.\]
The following three lemmas summarize some basic facts about this field.
\begin{lem}\label{froggy}
Let $p \equiv 2 \pmod 3$, and $K = \Q_p(\sqrt{-3})$.  For $x \in K^\times$, $|x|_p \in p^\Z$.
\end{lem}

\begin{proof}
Let $x = a+b\sqrt{-3}$, with $a, b \in \Q_p$ and $x \neq 0$.  Suppose $|a|_p \neq |b|_p$.  Then 
\[|x|_p = \left| a^2+3b^2\right|_p^{1/2} = \max\{|a|_p, |b|_p\} \in p^\Z.\]
Suppose instead that $|a|_p = |b|_p = p^\ell$.  Set $a_0 = p^\ell a$ and $b_0 = p^\ell b$.  Note that since $|a_0|_p=|b_0|_p=1$, we have $|a_0|_p, |b_0|_p \in p^\Z$.
Thus, 
\[ |a_0^2+3b_0^2|_p \leq \max \{1, |3|_p \} \leq 1.\]

Suppose, for the sake of contradiction, that $|a_0^2+3b_0^2|_p < 1$. 
Then we have that $a_0^2+3b_0^2 \equiv 0 \pmod p$, which is a contradiction since $-3$ is not a quadratic residue modulo $p$.
Thus
\[|x|_p = \left| a^2+3b^2\right|_p^{1/2} 
= \left| p^{-2\ell}(a_0^2+3b_0^2)\right|_p^{1/2} 
= p^{\ell}|a_0^2+3b_0^2|_p^{1/2} = p^{\ell}\in p^\Z. \qedhere\]
\end{proof}

\begin{lem}  \label{iscubeink}
Let $p$ be a prime with $p \equiv 2 \pmod 3$, $K = \Q_p (\sqrt{-3})$, and $C \in K$.  Let $k \in \N$, $p \nmid k$.  Then $f(x) = x^k-C$ has a root in $K$ if and only if
\begin{itemize}
\item[\textbf{(a)}] $|C|_p = p^{k\ell}$ for some $l \in \Z$, and 
\item[\textbf{(b)}] $p^{k\ell}C \pmod p$ is a $k^{\mathrm{th}}$ power in $\Z_p[\sqrt{-3}]/(p)$.
\end{itemize}
\end{lem}

\begin{proof}
First we assume the existence of $ r \in K$ so that $f(r)=0$, and verify that \textbf{(a)} and \textbf{(b)} hold.
By Lemma \ref{froggy}, $|r|_p = p^\ell$ for some $\ell \in \Z$.  Since
\[|C|_p = |r^k|_p= p^{k\ell},\] \textbf{(a)} is true. 
Further, 
\[p^{kl}C = p^{kl}r^k = (p^l r)^k\]
and thus $p^{kl}C$
 is the $k^\mathrm{th}$ power of $p^l r \pmod p$ in $\Z[\sqrt{-3}]$, and therefore also holds after reduction modulo $(p)$.

Conversely, we suppose $C \in \Q_p(\sqrt{-3})$ satisfies conditions \textbf{(a)} and \textbf{(b)}, and show that $C$ is a $k^{\mathrm{th}}$ power in $K$.  Replacing $C$ with $p^{kl}C$, without loss of generality we may assume $|C|_p=1$.  By condition \textbf{(b)}, there exists $a+b\sqrt{-3} \in \Z_p[\sqrt{-3}]/(p)$, where $a, b \in \{0, 1, 2, \cdots, p-1\}$ and  $C \equiv (a+b\sqrt{-3})^k \pmod p$.
Then 
\begin{align*}
 \left| f(a+b\sqrt{-3})  \right|_p &= \left|(a+b\sqrt{-3}) ^k -  C \right|_p \leq \tfrac{1}{p}, \text{ and }\\
 \left| f'(a+b\sqrt{-3})  \right|_p &= \left|  k (a+b\sqrt{-3}) ^{k-1} \right|_p =1.
\end{align*}
Thus, by Hensel's Lemma $f$ has a root in $K$.
\end{proof}

\begin{lem}
Let $p$ be a prime with $p \equiv 2 \pmod 3$, and $K = \Q_p (\sqrt{-3})$.  Let $x \in \Q_p$ be nonzero and the square of an element in $K$.  Then exactly one of the following two cases is true:
\begin{itemize}
\item[\textbf{(a)}] $x = a^2$ for some $a \in \Q_p$, or
\item[\textbf{(b)}] $x=-3b^2$ for some $b \in \Q_p$.
\end{itemize}
\end{lem}

\begin{proof}
Suppose $x=(a+b\sqrt{-3})^2$ for $a, b \in \Q_p$.  Then 
$x =a^2-3b^2+2ab\sqrt{-3}$.
Since $\sqrt{-3} \notin \Q_p$, we have $ab=0$.  If $a=0$, then $x=-3b^2$ and \textbf{(b)} holds.  If $b=0$, then $x=a^2$ and \textbf{(a)} holds.
\end{proof}

The previous lemma gives us the machinery to detect and solve for a square root in $K$, since $x$ is a square in $K$ and not in $\Q_p$ if and only if $\frac{x}{-3} = b^2$ for some $b \in \Q_p$.  

\begin{thm} Let $p$ be an odd prime, with $p \equiv 2 \pmod 3$.  Then the following algorithm yields $\tau_{3,p}$.  
\begin{itemize}
\item[\textbf{(1)}] Create a list, in ascending order of Mahler measure, of all irreducible, non-cyclotomic cubic polynomials in $\Z[x]$ with Mahler measure less than $8.5$.  Let $f(x)$ be the first polynomial on the list.  
\item[\textbf{(2)}] Convert $f(x)$ into depressed form $g(x) = x^3+ Ax+B$ and let $\Delta =B^2+4A^3/27 $.  
\item[\textbf{(3)}] If $\Delta$ is a square in $\Q_p$ or is not a square in $\Q_p(\sqrt{-3})$, return to step {\bf (2)} with the next polynomial on the list.
\item[\textbf{(4)}] If $A=0$, let $C=-B$, and otherwise let $C$ be a square root of $\Delta$ in $\Q_p(\sqrt{-3})$.  If $\frac{-B+C}{2}$ is not a cube in $\Q_p(\sqrt{-3})$, return to step {\bf (2)} with the next polynomial on the list.  
\item[\textbf{(5)}] If $\frac{-B+C}{2}$ is a cube in $\Q_p$, return to step {\bf (2)} with the next polynomial on the list. 
Otherwise, terminate, $\tau_{3,p} = \tfrac{1}{3} \log M(f)$. 
\end{itemize}
\end{thm}

\begin{proof}
Since $\tau_{3,p} \leq \tau_{3,p}^\ab$, by Theorem \ref{N3} we know that $\tau_{3,p} \leq 0.70376$.  By \cite[Proposition $1.6.7$]{bombierigubler}, a list of all polynomials with length less than $68$ will contain all irreducible, non-cyclotomic, cubic polynomials with Mahler measure bounded above by $8.5$.
Any degree $3$ algebraic number of height less than or equal to $0.70376$ will be a root of a polynomial in the list.  Thus, this algorithm will always terminate successfully.

Let $f$ be the polynomial being considered.
By Theorem \ref{twomodthreethm}, steps \textbf{(3)}, \textbf{(4)}, and \textbf{(5)} will detect exactly when $f$ splits completely over $\Q_p$.
\end{proof}
\newpage
\subsection{Implementation}\label{section2}
The function \textbf{is\_cube\_in\_k} checks to see if $A+B\sqrt{-3}$ is a cube in $K = \Q_p(\sqrt{-3})$ by applying Lemma \ref{iscubeink}. 
\begin{small}

\begin{Verbatim}[frame=single]
def is_cube_in_k(A,B,p):
    A = K(A);
    B = K(B);
    AA = A.list();
    BB = B.list();
    A0 = AA[0];
    B0 = BB[0];    
    if A.abs()<1:
        A0 = 0
    if B.abs()<1:
        B0 = 0
    for c in [0..p-1]:
        for d in [0..p-1]:
            if (c*c*c - 9*c*d*d)%p==A0:
                if (3*c*c*d - 3*d*d*d)%p==B0:
                    return True
    return False
\end{Verbatim}
\end{small}

The function \textbf{is\_cube\_in\_Qp} checks to see if $A$ is a cube in $\Q_p$ by applying Lemma \ref{lem2}.
 \begin{small}

\begin{Verbatim}[frame=single]   
def is_cube_in_Qp(A,p):
    val = A.ordp();
    if 3.divides(val)==True:
        L = A.expansion();
        a = L[0];
        if IsCubeInFp(a,p)==True:
            return True;
    return False
\end{Verbatim}
\end{small}

The function \textbf{tau\_dp\_1mod3} determines $\tau_{3,p}$ for the prime $p$ where $p \equiv 1 \pmod 3$, by implementing the algorithm described in Theorem \ref{Thm1mod3}.  Recall the array \textbf{Polynomials} contains the contents of the file \textbf{irred\_polynomials\_L68}, 
which has $L$ entries. 
These were calculated in Section $2.1$.
\begin{small}

\begin{Verbatim}[frame=single]
def tau_dp_1mod3(p):
    i = 0;
    while i < L-1:
        A = Polynomials[i][5];
        B = Polynomials[i][6];
        D = Polynomials[i][7];
        A = K(A);
        B = K(B);
        D = K(D);        
        if QQ(D).is_padic_square(p)==True:
            if A==0:
                C = -B;
            if A!=0:
                C = D.square_root();
            Check = (C - B) / 2;
            if is_cube_in_Qp(Check,p)==True:
                return Polynomials[i]
        i = i + 1;
    return False
\end{Verbatim}
\end{small}

The function \textbf{tau\_dp\_2mod3} determines $\tau_{3,p}$ for the prime $p$ where $p \equiv 2 \pmod 3$, by implementing the algorithm described in Theorem $12$.  
\begin{small}

\begin{Verbatim}[frame=single]
def tau_dp_2mod3(p):
    i = 0;
    while i < L-1:
        D = Polynomials[i][7];        
        if D.is_padic_square(p)==False:
            b = D / (-3);            
            if b.is_padic_square(p)==True:
                a = - Polynomials[i][6] / 2;
                b = K(b);
                b = sqrt(b) / 2;               
                if is_cube_in_k(a,b,p)==True:
                    return Polynomials[i]
        i=i+1;
    return False
\end{Verbatim}
\end{small}

The following code determines $\tau_{3,p}$ for all primes $p$ greater than $5$, up to and including the $N^\mathrm{th}$ prime. 
\begin{small}

\begin{Verbatim}[frame=single]
Polynomials=load('irred_polynomials_L68')
L=len(Polynomials)
P=Primes(); # P is now a list of all primes
N=25
rows = [['P', '$\tau_{3,p}$', 'Polynomial']]

for i in[2..N]:
    p = P.unrank(i);
    K = Qp(p, prec = 6, type = 'capped-rel', print_mode = 'series');    
    if p%3==1:
        tdp = tau_dp_1mod3(p)
        Poly = tdp[1]*x^3 + tdp[2]*x^2 + tdp[3]*x + tdp[4];
        h = tdp[8].n(digits=5);
        rows.append([p,h,Poly])       
    if p%3==2:
        tdp = tau_dp_2mod3(p)
        Poly = tdp[1]*x^3 + tdp[2]*x^2 + tdp[3]*x +tdp[4];
        h = tdp[8].n(digits=5);
        rows.append([p,h,Poly])
\end{Verbatim}
\end{small}
 
\newpage
\renewcommand*{\arraystretch}{1.0}
\section{Results}\label{results}
The table below contains some values for $\tau_{3,p}$.\\
\vspace{3mm}\\
\begin{minipage}{0.48\textwidth}
\begin{center}
\begin{tabular}{|c|c|c|}
\hline
$p$ & $\tau_{3,p}$ & $f_\alpha$ \\ \hline
$5$ & $0.36620$ & $x^{3} - 2 \, x^{2} - x - 3$ \\ \hline
$7$ & $0.30387$ & $2 \, x^{3} - 2 \, x^{2} + x - 2$ \\ \hline
$11$ & $0.36620$ & $x^{3} - x^{2} - 2 \, x - 3$ \\ \hline
$13$ & $0.26986$ & $x^{3} - 2 \, x^{2} - x + 1$ \\ \hline
$17$ & $0.23105$ & $x^{3} - x^{2} - x + 2$ \\ \hline
$19$ & $0.23105$ & $x^{3} - x^{2} - 2$ \\ \hline
$23$ & $0.23105$ & $x^{3} - x^{2} + x - 2$ \\ \hline
$29$ & $0.26986$ & $x^{3} - 2 \, x^{2} - x + 1$ \\ \hline
$31$ & $0.23105$ & $x^{3} - x - 2$ \\ \hline
$37$ & $0.27319$ & $x^{3} - x^{2} - 2 \, x - 2$ \\ \hline
$41$ & $0.23105$ & $x^{3} - x^{2} + x - 2$ \\ \hline
$43$ & $0.23105$ & $x^{3} - 2$ \\ \hline
$47$ & $0.12741$ & $x^{3} - x^{2} - 1$ \\ \hline
$53$ & $0.20313$ & $x^{3} - x^{2} - x - 1$ \\ \hline
$59$ & $0.093733$ & $x^{3} - x^{2} + 1$ \\ \hline
$61$ & $0.28206$ & $2 \, x^{3} - x^{2} + 2$ \\ \hline
$67$ & $0.12741$ & $x^{3} - x^{2} - 1$ \\ \hline
$71$ & $0.23105$ & $x^{3} - x^{2} - x + 2$ \\ \hline
$73$ & $0.29111$ & $2 \, x^{3} - x^{2} - 2$ \\ \hline
$79$ & $0.28612$ & $x^{3} - 2 \, x^{2} - 2$ \\ \hline
$83$ & $0.23105$ & $x^{3} - 2 \, x - 2$ \\ \hline
$89$ & $0.27535$ & $2 \, x^{3} - 2 \, x^{2} - x + 2$ \\ \hline
$97$ & $0.26986$ & $x^{3} - 2 \, x^{2} - x + 1$ \\ \hline
$101$ & $0.093733$ & $x^{3} - x^{2} + 1$ \\ \hline
$103$ & $0.20313$ & $x^{3} - x^{2} - x - 1$ \\ \hline
$107$ & $0.23105$ & $x^{3} - x - 2$ \\ \hline
$109$ & $0.23105$ & $x^{3} - 2$ \\ \hline
$113$ & $0.23105$ & $x^{3} - x - 2$ \\ \hline
$127$ & $0.23105$ & $x^{3} - x^{2} - 2$ \\ \hline
$131$ & $0.12741$ & $x^{3} - x^{2} - 1$ \\ \hline
$137$ & $0.30697$ & $x^{3} - x^{2} - 3 \, x - 2$ \\ \hline
$139$ & $0.23105$ & $x^{3} - x^{2} - x + 2$ \\ \hline
$149$ & $0.12741$ & $x^{3} - x^{2} - 1$ \\ \hline
$151$ & $0.28206$ & $2 \, x^{3} - x^{2} + 2$ \\ \hline
$157$ & $0.23105$ & $x^{3} - 2 \, x - 2$ \\ \hline
$163$ & $0.20313$ & $x^{3} - x^{2} - x - 1$ \\ \hline
$167$ & $0.093733$ & $x^{3} - x^{2} + 1$ \\ \hline
$173$ & $0.093733$ & $x^{3} - x^{2} + 1$ \\ \hline
$179$ & $0.27319$ & $x^{3} - x^{2} - 2 \, x - 2$ \\ \hline
$181$ & $0.26986$ & $x^{3} - 2 \, x^{2} - x + 1$ \\ \hline
$191$ & $0.23105$ & $x^{3} - x^{2} - 2$ \\ \hline
$193$ & $0.23105$ & $x^{3} - x^{2} + x - 2$ \\ \hline

\end{tabular}
\end{center}
\end{minipage}
\begin{minipage}{0.48\textwidth}
\begin{center}
\begin{tabular}{|c|c|c|}
\hline
$p$ & $\tau_{3,p}$ & $f_\alpha$ \\ \hline
$197$ & $0.23105$ & $x^{3} - x^{2} - x + 2$ \\ \hline
$199$ & $0.20313$ & $x^{3} - x^{2} - x - 1$ \\ \hline
$211$ & $0.093733$ & $x^{3} - x^{2} + 1$ \\ \hline
$223$ & $0.093733$ & $x^{3} - x^{2} + 1$ \\ \hline
$227$ & $0.12741$ & $x^{3} - x^{2} - 1$ \\ \hline
$229$ & $0.23105$ & $x^{3} - x^{2} + x - 2$ \\ \hline
$233$ & $0.27319$ & $x^{3} - x^{2} - 2 \, x - 2$ \\ \hline
$239$ & $0.26986$ & $x^{3} - 2 \, x^{2} - x + 1$ \\ \hline
$241$ & $0.30697$ & $x^{3} - x^{2} - 3 \, x - 2$ \\ \hline
$251$ & $0.23105$ & $x^{3} - x - 2$ \\ \hline
$257$ & $0.20313$ & $x^{3} - x^{2} - x - 1$ \\ \hline
$263$ & $0.27319$ & $x^{3} - x^{2} - 2 \, x - 2$ \\ \hline
$269$ & $0.20313$ & $x^{3} - x^{2} - x - 1$ \\ \hline
$271$ & $0.093733$ & $x^{3} - x^{2} + 1$ \\ \hline
$277$ & $0.23105$ & $x^{3} - x^{2} - 2$ \\ \hline
$281$ & $0.26986$ & $x^{3} - 2 \, x^{2} - x + 1$ \\ \hline
$283$ & $0.12741$ & $x^{3} - x^{2} - 1$ \\ \hline
$293$ & $0.12741$ & $x^{3} - x^{2} - 1$ \\ \hline
$307$ & $0.093733$ & $x^{3} - x^{2} + 1$ \\ \hline
$311$ & $0.20313$ & $x^{3} - x^{2} - x - 1$ \\ \hline
$313$ & $0.23105$ & $x^{3} - 2 \, x - 2$ \\ \hline
$317$ & $0.093733$ & $x^{3} - x^{2} + 1$ \\ \hline
$331$ & $0.28206$ & $2 \, x^{3} - x^{2} + 2$ \\ \hline
$337$ & $0.26986$ & $x^{3} - 2 \, x^{2} - x + 1$ \\ \hline
$347$ & $0.093733$ & $x^{3} - x^{2} + 1$ \\ \hline
$349$ & $0.12741$ & $x^{3} - x^{2} - 1$ \\ \hline
$353$ & $0.23105$ & $x^{3} - x^{2} - 2$ \\ \hline
$359$ & $0.23105$ & $x^{3} - x - 2$ \\ \hline
$367$ & $0.23105$ & $x^{3} - x^{2} - 2$ \\ \hline
$373$ & $0.23105$ & $x^{3} - x^{2} - x + 2$ \\ \hline
$379$ & $0.12741$ & $x^{3} - x^{2} - 1$ \\ \hline
$383$ & $0.23105$ & $x^{3} - x^{2} - x + 2$ \\ \hline
$389$ & $0.23105$ & $x^{3} - x^{2} - x + 2$ \\ \hline
$397$ & $0.20313$ & $x^{3} - x^{2} - x - 1$ \\ \hline
$401$ & $0.20313$ & $x^{3} - x^{2} - x - 1$ \\ \hline
$409$ & $0.30387$ & $2 \, x^{3} - 2 \, x^{2} + x - 2$ \\ \hline
$419$ & $0.20313$ & $x^{3} - x^{2} - x - 1$ \\ \hline
$421$ & $0.20313$ & $x^{3} - x^{2} - x - 1$ \\ \hline
$431$ & $0.12741$ & $x^{3} - x^{2} - 1$ \\ \hline
$433$ & $0.23105$ & $x^{3} - 2$ \\ \hline
$439$ & $0.23105$ & $x^{3} - x^{2} - x + 2$ \\ \hline
$443$ & $0.23105$ & $x^{3} - x^{2} - 2$ \\ \hline
\end{tabular}
\end{center}
\end{minipage}

\newpage
\section{Conclusion \& Future Work}\label{conclusion}
In this paper we relied on the fact that we can determine that a finite list of polynomials is guaranteed to contain one that splits over $\Q_p$ for any prime $p$. We restricted our search to cubic numbers that exist in abelian extensions of $\Q$ to prove this. Moving forward, we will determine that we can guarantee that for any degree $d$, there is some $N_d \in \Z$ such that $\tau_{d,p}^\ab$ depends only on $p \pmod {N_d}$. For example, $N_2 = 5$ and $N_3 = 228979643050431$.

When we look at the small nonzero values attained by the height function on cubic numbers, we see that the smallest value is $0.093733$. It would be interesting to classify all primes such that $\tau_{3,p} = 0.093733$.

%\bibliographystyle{alpha}
%\bibliography{ANTS}

\end{document}